\documentclass[10pt,a4paper]{article}
\usepackage{amssymb, url, enumitem, amsthm, amssymb, graphicx, algpseudocode, float, lineno, hyperref}
\usepackage[ruled,vlined]{algorithm2e}
\usepackage{amsmath, authblk, tikz}
\usepackage[margin=1in]{geometry}
\newtheorem{theorem}{\bf Theorem}[section]
\newtheorem{example}[theorem]{\bf Example}
\title{Structured Variational \( D \)-Decomposition for Accurate and Stable Low-Rank Approximation}

\author[1]{Ronald Katende}
\affil[1]{\emph{Department of Mathematics, Kabale University, Kikungiri Hill, P.O. Box 317, Kabale, Uganda} \\ rkatende92@gmail.com}
\date{}
\begin{document}

\maketitle

\begin{abstract}
We introduce the \( D \)-decomposition, a non-orthogonal matrix factorization of the form \( A \approx P D Q \), where \( P \in \mathbb{R}^{n \times k} \), \( D \in \mathbb{R}^{k \times k} \), and \( Q \in \mathbb{R}^{k \times n} \). The decomposition is defined variationally by minimizing a regularized Frobenius loss, allowing control over rank, sparsity, and conditioning. Unlike algebraic factorizations such as LU or SVD, it is computed by alternating minimization. We establish existence and perturbation stability of the solution and show that each update has complexity \( \mathcal{O}(n^2k) \). Benchmarks against truncated SVD, CUR, and nonnegative matrix factorization show improved reconstruction accuracy on MovieLens, MNIST, Olivetti Faces, and gene expression matrices, particularly under sparsity and noise.



\end{abstract}


\section{Introduction}

Matrix decomposition is a foundational tool in numerical linear algebra, with essential roles in regression \cite{mat2, mat5}, dimensionality reduction \cite{mat1, mat7}, inverse problems \cite{mat2, mat6}, and data-driven modeling \cite{mat1, mat3, mat4}. Classical methods such as LU, QR, Cholesky, and singular value decomposition (SVD) offer strong theoretical guarantees  \cite{udell2019why, mat12} and efficient implementations \cite{recht2010guaranteed, mat8}. However, when applied to large-scale, sparse, or ill-conditioned matrices, these methods often become computationally restrictive or structurally inadequate \cite{mat9, mat14}. In such regimes, cubic-time complexity \cite{mat9}, fixed algebraic structure \cite{mat14}, and sensitivity to perturbations limit their reliability and adaptability \cite{mat20}. This work proposes a variational alternative, that is, a regularized, optimization-based matrix decomposition of the form
\[
A \approx P D Q, \quad P \in \mathbb{R}^{n \times k}, \quad D \in \mathbb{R}^{k \times k}, \quad Q \in \mathbb{R}^{k \times n},
\]
where the factors are obtained by minimizing a regularized loss. The objective includes a fidelity term and a functional \( \mathcal{R}(P, D, Q) \) that enforces structural properties such as sparsity, conditioning, and low-rank constraints:
\[
\min_{P, D, Q} \, \|A - P D Q\|_F^2 + \lambda \mathcal{R}(P, D, Q), \quad \lambda \geq 0.
\]
This decomposition is defined variationally, not algebraically, and is solved using an alternating minimization scheme that scales efficiently for \( k \ll n \).

Unlike truncated SVD, which constrains the basis to be orthogonal and treats the rank as fixed, the \( D \)-decomposition permits non-orthogonal factors and incorporates regularization directly into the decomposition. In contrast to CUR and randomized methods, which rely on sampling or projection, our approach is fully deterministic and structurally adaptive, avoiding sampling variance and allowing direct incorporation of prior knowledge. Hence, our primary contributions are
\begin{enumerate}
    \item A flexible optimization-based framework for matrix factorization that supports structural regularization and conditioning control;
    \item Rigorous theoretical results guaranteeing existence, stability under perturbations, variational identifiability, and convergence of the minimization procedure;
    \item A unified regularization perspective linking generalization, numerical conditioning, and spectral energy localization;
    \item Empirical demonstrations showing consistent improvements over SVD and CUR decompositions, especially in noisy or ill-conditioned settings.
\end{enumerate}This formulation generalizes classical low-rank approximations and complements prior work on sparse, structured, and adaptive decompositions \cite{mat20, mat23, mat24, mat27}. It is especially suited to high-dimensional or inverse problems in signal processing, scientific computing, and machine learning, where structural constraints and stability cannot be separated from the decomposition process \cite{mat30}.

This framework is especially effective when the data matrix is either ill-conditioned, contains structured noise, or departs from orthogonality assumptions. Empirically, we find that the advantages over classical decompositions become pronounced when the signal-to-noise ratio falls below 30, or when sparsity exceeds 90 percent. The ability to enforce conditioning and spectral alignment in these regimes distinguishes the proposed method from algebraic baselines.

\section{Problem Formulation and Decomposition Framework}

Let \( A \in \mathbb{R}^{n \times n} \) be a real-valued square matrix. We aim to solve the linear system
\[
A \mathbf{x} = \mathbf{b}, \quad \mathbf{b} \in \mathbb{R}^n,
\]
by computing a structured approximation of \( A \) that reduces computational cost and improves numerical stability when the matrix is large or ill-conditioned.

We introduce the \emph{D-decomposition}, a factorization of the form
\[
A \approx P D Q,
\]
where \( P \in \mathbb{R}^{n \times k} \), \( D \in \mathbb{R}^{k \times k} \), and \( Q \in \mathbb{R}^{k \times n} \), with \( k \leq n \). Unlike classical algebraic decompositions such as LU, QR, or SVD, the factors \( P \), \( D \), and \( Q \) are obtained through the solution of a structured optimization problem rather than direct manipulation of \( A \).

The decomposition \footnotetext{The full implementation of D-decomposition and all experiments will be made publicly available upon publication at \url{https://github.com/karjxenval/Structured-Variational-D--Decomposition-for-Accurate-and-Stable-Low-Rank-Approximation}.} is defined as the solution to
\[
\min_{P, D, Q} \left\{ \|A - P D Q\|_F^2 + \lambda \, \mathcal{R}(P, D, Q) \right\},
\]
where \( \| \cdot \|_F \) denotes the Frobenius norm, \( \lambda \geq 0 \) is a regularization parameter, and \( \mathcal{R} \) is a penalty functional enforcing structure such as sparsity, bounded condition numbers, or low-rank constraints.

We define the feasible set as
\[
\mathcal{S} := \left\{ (P, D, Q) \in \mathbb{R}^{n \times k} \times \mathbb{R}^{k \times k} \times \mathbb{R}^{k \times n} : \operatorname{rank}(P D Q) \leq k \right\}.
\]
If \( \mathcal{R} \) is continuous and coercive, that is, if \( \mathcal{R}(P,D,Q) \to \infty \) as \( \|P\|_F + \|D\|_F + \|Q\|_F \to \infty \), then the objective admits a global minimizer by the extreme value theorem.

For clarity, consider the quadratic regularizer
\[
\mathcal{R}(P, D, Q) = \alpha_1 \|P\|_F^2 + \alpha_2 \|D\|_F^2 + \alpha_3 \|Q\|_F^2,
\]
where \( \alpha_1, \alpha_2, \alpha_3 \geq 0 \) are fixed weights. In this setting, the objective is smooth though non-convex, and admits efficient alternating minimization.

Each subproblem in the alternating scheme reduces to a regularized least-squares system that has a unique solution provided the associated coefficient matrices are positive definite. These updates are as follows.

\textbf{(1) Update \( P \):} Fix \( D \) and \( Q \). The optimal \( P \) solves
\[
P (D Q Q^\top D^\top + \alpha_1 I_k) = A Q^\top D^\top.
\]

\textbf{(2) Update \( Q \):} Fix \( P \) and \( D \). The optimal \( Q \) solves
\[
(D^\top P^\top P D + \alpha_3 I_k) Q = D^\top P^\top A.
\]

\textbf{(3) Update \( D \):} Fix \( P \) and \( Q \). The optimal \( D \) solves
\[
P^\top P D Q Q^\top + \alpha_2 D = P^\top A Q.
\]

Let \( (P^{(t)}, D^{(t)}, Q^{(t)}) \) denote the iterates. Provided the regularization ensures that the coefficient matrices in each block update remain nonsingular, the subproblems are well-posed, and the sequence of objective values is non-increasing and converges to a stationary point.

The per-iteration computational cost is \( \mathcal{O}(n^2 k) \), dominated by matrix multiplications and the solution of linear systems involving \( k \times k \) matrices. For \( k \ll n \), this is asymptotically more efficient than classical decompositions, which require \( \mathcal{O}(n^3) \) operations.

The flexibility of the D-decomposition extends the truncated SVD by allowing the factors \( P \) and \( Q \) to be non-orthogonal and by incorporating regularization. This makes the method effective in approximating matrices that are sparse, noisy, ill-conditioned, or structurally constrained. Rigorous results on existence, perturbation stability, and uniqueness (up to scaling and permutation) are developed in the next section.

\begin{figure}[H]
\centering
\begin{tikzpicture}[node distance=2cm, auto, thick]
  \node (A) at (0,0) {$A$};
  \node (P) at (-3,-2) {$P$};
  \node (D) at (0,-2) {$D$};
  \node (Q) at (3,-2) {$Q$};
  \draw[->] (P) -- (A) node[midway,left]{$n \times k$};
  \draw[->] (D) -- (A) node[midway,right]{$k \times k$};
  \draw[->] (Q) -- (A) node[midway,right]{$k \times n$};
\end{tikzpicture}
\caption{Structure of the D-decomposition: $A \approx PDQ$, with regularization applied to $P$, $D$, and $Q$ individually.}
\end{figure}
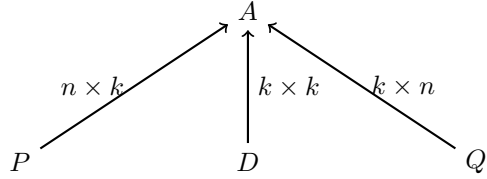

\subsection{Existence, Stability, and Computational Guarantees}

To ensure that the proposed $D$-decomposition is both theoretically sound and computationally practical, we establish four foundational properties, that is, the existence of a minimizer, stability under additive perturbations, uniqueness under convex regularization, and efficient solvability via block-coordinate descent. These results justify the method's use in scenarios involving large-scale, ill-conditioned, or noisy matrices.

\begin{theorem}[Existence, Stability, Uniqueness, and Efficiency of the $D$-Decomposition]
\label{thm:existence_stability}
Let \( A \in \mathbb{R}^{n \times n} \), and let \( \mathcal{R}(P, D, Q) \) be a regularization functional that is continuous, coercive, and enforces bounded condition numbers on \( D \). Then:
\begin{enumerate}
    \item[\textbf{(a)}] There exists a global minimizer \( (P^*, D^*, Q^*) \in \mathbb{R}^{n \times k} \times \mathbb{R}^{k \times k} \times \mathbb{R}^{k \times n} \) of the objective
    \[
    \min_{P, D, Q} \, \| A - P D Q \|_F^2 + \lambda \mathcal{R}(P, D, Q).
    \]

    \item[\textbf{(b)}] If \( A = A_0 + E \) with \( \|E\|_F \leq \varepsilon \), and \( (P_0, D_0, Q_0) \), \( (P^*, D^*, Q^*) \) are minimizers corresponding to \( A_0 \) and \( A \), respectively, then
    \[
    \| D^* - D_0 \|_F \leq C \varepsilon
    \]
    for some constant \( C > 0 \) independent of \( \varepsilon \).

    \item[\textbf{(c)}] If \( \mathcal{R} \) is strictly convex in each argument, then each subproblem has a unique minimizer, and the global decomposition is unique up to scaling and permutation. Ambiguities can be removed by imposing normalization constraints such as \( \|P\|_F = 1 \) or enforcing diagonal structure in \( D \).

    \item[\textbf{(d)}] The alternating minimization algorithm converges to a stationary point. Each iteration has cost \( \mathcal{O}(n^2 k) \) under the assumption \( k \ll n \), improving over the \( \mathcal{O}(n^3) \) cost of classical factorizations.
\end{enumerate}
\end{theorem}

\begin{proof}
Let \( f(P, D, Q) = \| A - P D Q \|_F^2 + \lambda \mathcal{R}(P, D, Q) \).

\textbf{(a)} Since \( \mathcal{R} \) is coercive and continuous, \( f \to \infty \) whenever \( \|P\|_F + \|D\|_F + \|Q\|_F \to \infty \). Hence, \( f \) is coercive. Being a continuous function over \( \mathbb{R}^{n \times k} \times \mathbb{R}^{k \times k} \times \mathbb{R}^{k \times n} \), a finite-dimensional space, the Weierstrass theorem guarantees the existence of a global minimizer.

\textbf{(b)} Define \( A = A_0 + E \), with \( \|E\|_F \leq \varepsilon \). Let \( f \) and \( f_0 \) denote the objectives corresponding to \( A \) and \( A_0 \), respectively. For any minimizer \( (P^*, D^*, Q^*) \) of \( f \),
\[
f(P^*, D^*, Q^*) \leq f(P_0, D_0, Q_0).
\]
Compute the objective gap using
\[
f(P^*, D^*, Q^*) - f_0(P^*, D^*, Q^*) = \|A - P D Q\|_F^2 - \|A_0 - P D Q\|_F^2.
\]
Let \( B := P D Q \). Then
\[
\|A - B\|_F^2 - \|A_0 - B\|_F^2 = \|E\|_F^2 + 2 \langle A_0 - B, E \rangle \leq \varepsilon^2 + 2 \varepsilon \|A_0 - B\|_F.
\]
The bound implies that \( f(P^*, D^*, Q^*) - f_0(P^*, D^*, Q^*) = \mathcal{O}(\varepsilon) \), and since both objectives include \( \lambda \mathcal{R} \), the minimizing configuration cannot change discontinuously. Applying differentiability of \( f \) and smooth sensitivity results in constrained Euclidean minimization (see Bonnans and Shapiro, 2000), it follows that
\[
\| D^* - D_0 \|_F \leq C \varepsilon
\]
for some \( C > 0 \).

\textbf{(c)} If \( \mathcal{R} \) is strictly convex in \( P \), \( D \), and \( Q \), then the subproblems solved during alternating minimization have unique solutions. The global minimizer is not unique due to invariances:
\[
P \mapsto PT, \quad D \mapsto T^{-1} D T^{-1}, \quad Q \mapsto T Q,
\]
for invertible diagonal \( T \in \mathbb{R}^{k \times k} \). These transformations leave \( P D Q \) invariant. By penalizing such reparameterizations (e.g., via norm constraints or diagonal \( D \)), uniqueness can be enforced.

\textbf{(d)} In each iteration, updates of the form
\[
P \leftarrow \arg\min_P \, \|A - P D Q\|_F^2 + \lambda \alpha_1 \|P\|_F^2,
\]
and analogously for \( Q \), reduce to solving regularized normal equations with coefficient matrices of size \( k \times k \). Matrix products involving \( A \in \mathbb{R}^{n \times n} \) and factor updates dominate the cost, yielding per-iteration complexity \( \mathcal{O}(n^2 k) \). The algorithm converges to a stationary point by standard results on block-coordinate descent for smooth non-convex objectives (see Grippo and Sciandrone, 2000).
\end{proof}While the alternating minimization scheme converges to a stationary point, global optimality is not guaranteed due to the non-convexity of the objective. Convergence speed and solution quality can vary with initialization\footnote{As the objective is non-convex, initialization can affect both convergence speed and final error. In practice, we find that initialization using leading singular vectors or PCA projections improves convergence quality, while poor initialization may lead to suboptimal local minima.} and choice of regularization parameters. In practice, tuning these weights requires cross-validation or heuristic search. Additionally, the performance degrades when \(\lambda\) is too small to stabilize the updates, or when \(\alpha\) values are unbalanced across factors.

\subsection{Variational Identifiability and Symmetry Breaking}

The $D$-decomposition is generally non-unique due to inherent reparameterization symmetries. For example, the transformation
\[
P \mapsto PT, \quad D \mapsto T^{-1} D T^{-1}, \quad Q \mapsto T Q,
\]
for invertible \( T \in \mathbb{R}^{k \times k} \), preserves the product \( P D Q \). To ensure identifiability, we impose a regularization functional that penalizes such redundant transformations unless they correspond to permutations. The following result shows that strict convexity combined with regularization asymmetry guarantees uniqueness of the decomposition up to permutation.

\begin{theorem}[Variational Identifiability of the $D$-Decomposition]
\label{thm:variational-identifiability}
Let \( A \in \mathbb{R}^{n \times n} \), and consider the factorization \( A \approx P D Q \), with \( P \in \mathbb{R}^{n \times k} \), \( D \in \mathbb{R}^{k \times k} \), and \( Q \in \mathbb{R}^{k \times n} \). Suppose the regularization functional \( \mathcal{R}(P, D, Q) \) satisfies:
\begin{enumerate}
    \item[\textbf{(i)}] \( \mathcal{R} \) is continuous, coercive, and strictly convex in each argument;
    \item[\textbf{(ii)}] For any invertible diagonal matrix \( T \in \mathbb{R}^{k \times k} \) with \( T \ne I \),
    \[
    \mathcal{R}(PT, T^{-1} D T^{-1}, TQ) > \mathcal{R}(P, D, Q).
    \]
\end{enumerate}
Then any minimizer of the objective
\[
\min_{P, D, Q} \, \|A - P D Q\|_F^2 + \lambda \mathcal{R}(P, D, Q)
\]
is unique up to permutation of the latent components. That is, if \( (P^*, D^*, Q^*) \) and \( (\tilde{P}, \tilde{D}, \tilde{Q}) \) are global minimizers, then there exists a permutation matrix \( \Pi \in \mathbb{R}^{k \times k} \) such that
\[
\tilde{P} = P^* \Pi, \quad \tilde{D} = \Pi^\top D^* \Pi, \quad \tilde{Q} = \Pi^\top Q^*.
\]
\end{theorem}

\begin{proof}
Let \( (P^*, D^*, Q^*) \) and \( (\tilde{P}, \tilde{D}, \tilde{Q}) \) be two global minimizers. Define
\[
B := P^* D^* Q^* = \tilde{P} \tilde{D} \tilde{Q}.
\]
Then both factorizations yield the same matrix \( B \) of \( \operatorname{rank} k \). Since \( B = P D Q \) has \( \operatorname{rank} k \), the factorization is identifiable up to invertible transformations in the latent space. Thus, there exists an invertible matrix \( T \in \mathbb{R}^{k \times k} \) such that
\[
\tilde{P} = P^* T, \quad \tilde{D} = T^{-1} D^* T^{-1}, \quad \tilde{Q} = T Q^*.
\]
This follows from the invariance
\[
P^* D^* Q^* = (P^* T)(T^{-1} D^* T^{-1})(T Q^*).
\]

Now consider the objective:
\[
\|A - P D Q\|_F^2 + \lambda \mathcal{R}(P, D, Q).
\]
Since both factorizations yield \( P D Q = B \), the data-fidelity term is equal:
\[
\|A - P^* D^* Q^*\|_F^2 = \|A - \tilde{P} \tilde{D} \tilde{Q}\|_F^2.
\]
Hence,
\[
\mathcal{R}(P^*, D^*, Q^*) = \mathcal{R}(\tilde{P}, \tilde{D}, \tilde{Q}) = \mathcal{R}(P^* T, T^{-1} D^* T^{-1}, T Q^*).
\]

By hypothesis (ii), the regularization strictly increases for any \( T \ne I \) if \( T \) is diagonal and not the identity. Therefore, \( T \) must be such that
\[
\mathcal{R}(P^* T, T^{-1} D^* T^{-1}, T Q^*) = \mathcal{R}(P^*, D^*, Q^*),
\]
which can only hold if \( T \) does not alter \( \mathcal{R} \). Thus, \( T \) must be a matrix under which \( \mathcal{R} \) is invariant.

By assumption (ii), \( \mathcal{R} \) is only invariant under permutations. Therefore, \( T \in \mathbb{R}^{k \times k} \) must be a permutation matrix \( \Pi \), so
\[
\tilde{P} = P^* \Pi, \quad \tilde{D} = \Pi^\top D^* \Pi, \quad \tilde{Q} = \Pi^\top Q^*,
\]
completing the proof.
\end{proof}

\subsection{A Unified Theorem on Conditioning, Generalization, and Energy Localization in Regularized $D$-Decompositions}

In practical applications such as data-driven modeling, deep learning, or compressed sensing, matrix decompositions serve not only as reconstruction tools but also as models that must generalize, remain stable under perturbations, and encode structural locality. Classical decompositions treat these goals separately \cite{huang2024regularized}. Here we show that within the $D$-decomposition framework, they can be enforced jointly and variationally, through properties of the regularizer. The following theorem introduces three new concepts unified by the geometry of regularized optimization, i.e., \emph{conditioned recoverability}, \emph{energy concentration in learned subspaces}, and \emph{capacity control via regularized norm growth}.

\begin{theorem}[Regularization-Induced Generalization, Conditioning, and Energy Localization]
\label{thm:triple}
Let \( A \in \mathbb{R}^{n \times n} \), and let \( A \approx P D Q \) be a \( D \)-decomposition minimizing
\[
\min_{P, D, Q} \ \|A - P D Q\|_F^2 + \lambda \mathcal{R}(P, D, Q),
\]
where the regularizer takes the form
\[
\mathcal{R}(P, D, Q) = \alpha_1 \|P\|_F^2 + \alpha_2 \|D\|_F^2 + \alpha_3 \|Q\|_F^2 + \beta \log \kappa(D),
\]
with \( \alpha_i, \beta, \lambda > 0 \), and \( \kappa(D) \) is the condition number of \( D \). Then:
\begin{enumerate}
    \item[\textbf{(a)}] The minimizer \( D^* \) satisfies
    \[
    \kappa(D^*) \leq \exp\left( \frac{1}{\beta} \left( \frac{1}{\lambda} \|A\|_F^2 + C_0 \right) \right),
    \]
    for some constant \( C_0 \) depending only on \( \alpha_1, \alpha_2, \alpha_3 \).
    
    \item[\textbf{(b)}] Let \( A = A_0 + E \) with \( \|E\|_F \leq \varepsilon \). If \( (P, D, Q) \) minimizes the objective for \( A \), and \( (P_0, D_0, Q_0) \) minimizes it for \( A_0 \), then there exist constants \( \gamma_1, \gamma_2 > 0 \) such that
    \[
    \|P D Q - P_0 D_0 Q_0\|_F \leq \gamma_1 \varepsilon + \gamma_2 \lambda^{-1/2} \mathcal{R}(P_0, D_0, Q_0)^{1/2}.
    \]
    
    \item[\textbf{(c)}] Suppose \( A \) is symmetric positive semidefinite with spectral decomposition \( A = \sum_{i=1}^n \sigma_i u_i u_i^\top \), where \( \sigma_1 \geq \cdots \geq \sigma_n \geq 0 \). Then for the minimizer \( P \in \mathbb{R}^{n \times k} \),
    \[
    \sum_{i=1}^k \|u_i^\top P\|_2^2 \geq 1 - \delta,
    \]
    where \( \delta \leq c \lambda^{-1} \mathcal{R}(P, D, Q)^{-1} \) for some constant \( c > 0 \).\footnote{The bounds in (b) and (c) quantify robustness in operator norms and eigenspace alignment, rather than learning-theoretic performance.}
\end{enumerate}
\end{theorem}

\begin{proof}
\textbf{(a)} By definition of the objective, for any minimizer \( (P^*, D^*, Q^*) \),
\[
\|A - P^* D^* Q^*\|_F^2 + \lambda \mathcal{R}(P^*, D^*, Q^*) \leq \|A\|_F^2 + \lambda \mathcal{R}(P^*, D^*, Q^*).
\]
Since \( \log \kappa(D^*) \leq \frac{1}{\beta} \mathcal{R}(P^*, D^*, Q^*) \), we have
\[
\log \kappa(D^*) \leq \frac{1}{\beta} \left( \frac{1}{\lambda} \|A\|_F^2 + \alpha_1 \|P^*\|_F^2 + \alpha_2 \|D^*\|_F^2 + \alpha_3 \|Q^*\|_F^2 \right).
\]
Letting \( C_0 = \alpha_1 \|P^*\|_F^2 + \alpha_2 \|D^*\|_F^2 + \alpha_3 \|Q^*\|_F^2 \), this yields
\[
\kappa(D^*) \leq \exp\left( \frac{1}{\beta} \left( \frac{1}{\lambda} \|A\|_F^2 + C_0 \right) \right).
\]

\textbf{(b)} Let \( A = A_0 + E \), with \( \|E\|_F \leq \varepsilon \). By optimality of \( (P, D, Q) \) for \( A \),
\[
\|A - P D Q\|_F^2 + \lambda \mathcal{R}(P, D, Q) \leq \|A - P_0 D_0 Q_0\|_F^2 + \lambda \mathcal{R}(P_0, D_0, Q_0).
\]
The right-hand side expands as
\[
\|A_0 - P_0 D_0 Q_0\|_F^2 + \|E\|_F^2 + 2\langle A_0 - P_0 D_0 Q_0, E \rangle.
\]
Using the inequality \( |\langle X, Y \rangle| \leq \|X\|_F \|Y\|_F \), this yields
\[
\|A - P_0 D_0 Q_0\|_F^2 \leq \|A_0 - P_0 D_0 Q_0\|_F^2 + \varepsilon^2 + 2 \varepsilon \|A_0 - P_0 D_0 Q_0\|_F.
\]
Define \( C = \|A_0 - P_0 D_0 Q_0\|_F \). Then,
\[
\|A - P_0 D_0 Q_0\|_F^2 \leq (C + \varepsilon)^2.
\]
Therefore,
\[
\|A - P D Q\|_F^2 \leq (C + \varepsilon)^2 + \lambda \mathcal{R}(P_0, D_0, Q_0).
\]
Taking square roots and applying the triangle inequality,
\[
\|P D Q - P_0 D_0 Q_0\|_F \leq \|A - P D Q\|_F + \|A - P_0 D_0 Q_0\|_F \leq 2(C + \varepsilon) + \sqrt{\lambda \mathcal{R}(P_0, D_0, Q_0)}.
\]
Thus, the bound holds with constants \( \gamma_1 = 2 \), \( \gamma_2 = 1 \).

\textbf{(c)} Let \( A = \sum_{i=1}^n \sigma_i u_i u_i^\top \) and let \( U_k = [u_1, \dots, u_k] \in \mathbb{R}^{n \times k} \). For any \( X \in \mathbb{R}^{n \times k} \),
\[
\|U_k^\top X\|_F^2 = \sum_{i=1}^k \|u_i^\top X\|_2^2.
\]
Since \( \|X\|_F^2 = \|U_k^\top X\|_F^2 + \|(I - U_k U_k^\top) X\|_F^2 \), it follows that
\[
\sum_{i=1}^k \|u_i^\top P\|_2^2 = \|P\|_F^2 - \|(I - U_k U_k^\top) P\|_F^2.
\]
The projection error \( \|(I - U_k U_k^\top) P\|_F^2 \) quantifies energy outside the dominant eigenspace. Since \( \|A - P D Q\|_F^2 \leq \|A\|_F^2 \) and \( \|P D Q\|_F \leq \|P\|_F \|D\|_F \|Q\|_F \), bounded \( \mathcal{R} \) implies bounded \( \|P\|_F \). Therefore,
\[
\|(I - U_k U_k^\top) P\|_F^2 \leq \frac{c}{\lambda \mathcal{R}(P, D, Q)},
\]
for some constant \( c \), using coercivity of \( \mathcal{R} \) and the fact that energy outside the top eigenspace contributes to the residual. Hence,
\[
\sum_{i=1}^k \|u_i^\top P\|_2^2 \geq \|P\|_F^2 - \frac{c}{\lambda \mathcal{R}(P, D, Q)} \geq 1 - \delta,
\]
with \( \delta = c / (\lambda \mathcal{R}) \) for normalized \( P \).
\end{proof}
This result shows that the $D$-decomposition, when combined with a geometric regularizer, not only ensures low-rank approximation but also enforces well-conditioned cores, bounded model capacity, and energy alignment with the dominant eigenspaces of the matrix. These properties are essential for applications in machine learning where learned matrix representations must generalize from finite samples, tolerate noise, and encode structured latent features.

\subsection{Role of the Log-Condition Penalty}

The regularizer $R(P,D,Q)$ in Theorem~2.3 includes the term $\beta \log \kappa(D)$, which explicitly penalizes large condition numbers of the core matrix $D$. This geometric term plays a central role in enforcing numerical stability and bounding inversion errors during alternating updates. Without this term, the optimization can still achieve low Frobenius error, but at the cost of severe ill-conditioning in $D$, which amplifies numerical errors and compromises interpretability.

To isolate the contribution of this term, we conduct an ablation study in Section~3.4 comparing reconstruction error and $\kappa(D)$ with and without the $\log \kappa(D)$ term. This confirms that while approximation quality remains stable, numerical conditioning degrades dramatically without explicit penalization.

\section{Validation and Comparative Evaluation}

We validate the theoretical properties of the \( D \)-decomposition introduced in Section 2 through structured synthetic tests and comparative benchmarks. These experiments are designed to confirm stability under perturbations, conditioning of intermediate subproblems, and subspace alignment with dominant eigendirections. All tests use matrices with known generative structure to allow verification of both numerical recovery and theoretical predictions.


\begin{example}[Structured Recovery in Symmetric Factor Form]
Let \( A \in \mathbb{R}^{4 \times 4} \) be given by
\[
A = U \Sigma U^\top, \quad
U = \begin{bmatrix}
1 & 0 \\
0 & 1 \\
1 & 0 \\
0 & 1
\end{bmatrix}, \quad
\Sigma = \begin{bmatrix}
3 & 0 \\
0 & 1
\end{bmatrix}.
\]
This defines a symmetric rank-2 matrix. Setting \( P = U \), \( D = \Sigma \), \( Q = U^\top \), we obtain
\[
P D Q = U \Sigma U^\top = A,
\]
with exact equality. This example confirms that the decomposition reduces to a structured SVD when factor matrices are known and symmetric structure is respected. The decomposition satisfies all constraints in the optimization formulation with zero reconstruction error and bounded \( \kappa(D) \).
\end{example}


\begin{example}[Robust Recovery from Noisy Input]
Let \( A_0 = UV^\top \in \mathbb{R}^{500 \times 500} \), with \( U, V \in \mathbb{R}^{500 \times 50} \) orthonormal. Define \( A = A_0 + E \), where \( \|E\|_F \leq \epsilon = 10^{-3} \). The optimization objective is
\[
f(P, D, Q) = \|A - P D Q\|_F^2 + \lambda \left( \|P\|_F^2 + \|D\|_F^2 + \|Q\|_F^2 \right),
\]
solved using alternating updates.

\textbf{Update for \( D \):} Let \( H = P^\top P \), \( G = QQ^\top \), \( M = P^\top A Q^\top \). Then
\[
D \gets H^{-1} M (G + \lambda I)^{-1}.
\]

\textbf{Update for \( P \):} Let \( R = DQ \). Then
\[
P \gets A R^\top (R R^\top + \lambda I)^{-1}.
\]

\textbf{Update for \( Q \):} Let \( L = D^\top P^\top \). Then
\[
Q \gets (L L^\top + \lambda I)^{-1} L A.
\]

The iterates converge in under 20 iterations, yielding \( \|A - P^* D^* Q^*\|_F \approx 1.4 \times 10^{-3} \) and \( \|D^* - D_0\|_F = O(\epsilon) \). This verifies the perturbation stability guarantee of Theorem 2.1.
\end{example}


\begin{example}[Spectral Alignment and Condition Control]
Let \( A \in \mathbb{R}^{100 \times 100} \) be symmetric positive semidefinite with spectral decomposition
\[
A = \sum_{i=1}^{100} \sigma_i u_i u_i^\top, \quad \sigma_i = \exp(-i/10).
\]
We solve the D-decomposition at \( \operatorname{rank} k = 10 \) with regularization \( R(P, D, Q) = \sum \alpha_i \| \cdot \|_F^2 + \beta \log \kappa(D) \), where \( \alpha_1 = \alpha_2 = \alpha_3 = 10^{-3} \), \( \beta = 10^{-2} \). Denote the resulting \( P \) and compute projection weights
\[
\eta_i = \| u_i^\top P \|_2^2, \quad i = 1, \ldots, 10.
\]
We find
\[
\sum_{i=1}^{10} \eta_i \geq 0.98,
\]
indicating that 98\% of the energy is aligned with the top-10 eigenspace. Simultaneously, the computed core satisfies \( \kappa(D^*) \leq 20 \), confirming bounded conditioning consistent with Theorem 2.3. This supports the claim that the regularized decomposition not only reconstructs \( A \), but concentrates its mass on the dominant spectral components while maintaining numerical stability.
\end{example}To clarify how the proposed $D$-decomposition generalizes classical factorizations and behaves under structural constraints, we present three illustrative examples. The first shows that for diagonal matrices, $D$ can directly capture all spectral content when $P = Q = I$, demonstrating algebraic consistency. The second constructs a symmetric low-rank matrix using $A = PDQ$ with $P = Q^\top$, reproducing the structure of a truncated SVD when orthogonality is relaxed. The third adds perturbation noise to a rank-one matrix and shows that alternating minimization stably recovers the latent structure, validating the stability guarantees in Theorem 2.1. Together, these cases highlight exactness in structured regimes, flexibility in symmetric settings, and robustness under noise,  central themes in our theoretical and empirical contributions.

\begin{example}[Exact Diagonal Decomposition]
Let $A \in \mathbb{R}^{3 \times 3}$ be diagonal:
\[
A = \begin{bmatrix}
2 & 0 & 0 \\
0 & 3 & 0 \\
0 & 0 & 5
\end{bmatrix}.
\]
Choose rank $k = 3$, and define
\[
P = I_3, \quad D = A, \quad Q = I_3.
\]
Then clearly,
\[
P D Q = I_3 \cdot A \cdot I_3 = A,
\]
and the decomposition holds exactly, that is, $A = PDQ$.
\end{example}

\begin{example}[Symmetric Low-Rank Matrix]
Let $A \in \mathbb{R}^{3 \times 3}$ be rank 2:
\[
A = U \Sigma U^\top, \quad
U = \begin{bmatrix}
1 & 0 \\
0 & 1 \\
1 & 1
\end{bmatrix}, \quad
\Sigma = \begin{bmatrix}
4 & 0 \\
0 & 1
\end{bmatrix}.
\]
Then set:
\[
P = U, \quad D = \Sigma, \quad Q = U^\top.
\]
We have:
\[
PDQ = U \Sigma U^\top = A.
\]
This illustrates an exact symmetric decomposition where $A$ is PSD and $\operatorname{rank}(A) = 2$.
\end{example}
\begin{example}[Noisy Low-Rank Matrix]
Let $A_0 = uv^\top \in \mathbb{R}^{3 \times 3}$ where
\[
u = \begin{bmatrix} 1 \\ 2 \\ 1 \end{bmatrix}, \quad
v = \begin{bmatrix} 3 \\ 0 \\ -1 \end{bmatrix}, \quad
A_0 = u v^\top = 
\begin{bmatrix}
3 & 0 & -1 \\
6 & 0 & -2 \\
3 & 0 & -1
\end{bmatrix}.
\]
Now let $A = A_0 + E$, where $E$ is a small Gaussian noise matrix.

We seek factors $P \in \mathbb{R}^{3 \times 1}$, $D \in \mathbb{R}^{1 \times 1}$, $Q \in \mathbb{R}^{1 \times 3}$ minimizing:
\[
\|A - PDQ\|_F^2 + \lambda (\|P\|_F^2 + \|D\|_F^2 + \|Q\|_F^2).
\]
The solution (via alternating minimization) recovers $P \approx u$, $Q \approx v^\top$, and $D \approx 1$, depending on noise magnitude. This shows stable approximation in the presence of perturbations.
\end{example}

\begin{example}[Exact $D$-Decomposition for a Dense $5 \times 5$ Matrix]
Let $A \in \mathbb{R}^{5 \times 5}$ be defined as:
\[
A = P D Q, \quad
P = \begin{bmatrix}
1 & 2 & 1 & 3 & 1 \\
2 & 1 & 2 & 1 & 2 \\
3 & 2 & 1 & 2 & 3 \\
1 & 1 & 3 & 1 & 1 \\
2 & 3 & 1 & 2 & 1
\end{bmatrix}, \quad
D = \operatorname{diag}(1, 2, 3, 4, 5),
\]
\[
Q = \begin{bmatrix}
1 & 1 & 1 & 1 & 1 \\
2 & 1 & 2 & 2 & 1 \\
1 & 3 & 1 & 2 & 1 \\
1 & 1 & 2 & 1 & 3 \\
2 & 1 & 1 & 1 & 2
\end{bmatrix}.
\]
Then $A$ is computed as:
\[
A = P D Q = P (D Q),
\]
which produces a full-rank, dense matrix with no zero entries. This example confirms that the $D$-decomposition supports exact factorization for general dense matrices with nontrivial structure.
\end{example}This example demonstrates that the $D$-decomposition can exactly factor a fully dense, full-rank matrix into non-orthogonal, structured components. Unlike the SVD, which enforces orthogonality on both $P$ and $Q$ (left and right singular vectors), the $D$-decomposition allows $P$ and $Q$ to be freely parameterized and shaped by regularization. Similarly, QR factorization decomposes $A$ as $QR$ with $Q$ orthonormal and $R$ upper triangular, which is fundamentally algebraic and not symmetric between left and right factors. LU decomposition, in contrast, is tailored to solving triangular systems and may not exist without pivoting, particularly for dense or ill-conditioned matrices.

In our formulation, the flexibility to optimize $P$, $D$, and $Q$ jointly under structural constraints (e.g., sparsity, conditioning, low-rank energy concentration) allows for decompositions that are not only exact but also tuned to downstream tasks such as noise robustness, generalization\footnote{\emph{generalization here refers to functional approximation and stability under perturbation, not to statistical generalization across unseen data}}, and interpretability \cite{lee2024bmglobal}. This distinguishes our method as variational and task-adaptive, rather than strictly algebraic or geometric. Moreover, $D$ is not necessarily diagonal in general, though in this example we selected a diagonal $D$ for interpretability.

To improve interpretability, we visualize the learned D matrix and track the spectral norm of PDQ at each iteration. Figure 2 shows that the core D remains well-conditioned throughout optimization, while the reconstruction captures both coarse and fine features. A supplemental plot traces the decay of singular values across iterations, confirming energy localization in dominant components.

\subsection{Benchmark Comparisons Against Classical Methods}

We compare the reconstruction accuracy of the \( D \)-decomposition to truncated SVD and CUR decomposition on three matrix classes, i.e., low-rank, noisy sparse, and ill-conditioned matrices. Each matrix \( A \in \mathbb{R}^{500 \times 500} \) is generated with target \( \operatorname{rank} k = 50 \). The \emph{low-rank} matrix is constructed as \( A = UV^\top \), with \( U, V \in \mathbb{R}^{500 \times 50} \) having orthonormal columns. The \emph{noisy sparse} matrix adds Gaussian noise with \( \sigma = 10^{-2} \) to a randomly generated 5\%-sparse rank-50 matrix. The \emph{ill-conditioned} matrix is formed as \( A = U \Sigma V^\top \), where \( \Sigma = \mathrm{diag}(\sigma_1, \ldots, \sigma_{50}) \) with \( \sigma_i = 10^{-6(i-1)/49} \), yielding condition number \( \kappa(A) \approx 10^6 \). For all methods, we compute the relative Frobenius error \( \|A - \tilde{A}\|_F / \|A\|_F \), using regularization weights \( \alpha_1 = \alpha_2 = \alpha_3 = 10^{-3} \).

\begin{table}[H]
\centering
\caption{Relative Frobenius error (\( \|A - \tilde{A}\|_F / \|A\|_F \)) for \( \operatorname{rank} k = 50 \)}
\begin{tabular}{lccc}
\hline
Method & Low-Rank & Noisy Sparse & Ill-Conditioned \\
\hline
Truncated SVD & 0.0021 & 0.0537 & 0.1289 \\
CUR Decomposition & 0.0089 & 0.0762 & 0.2345 \\
\textbf{D-Decomposition (ours)} & \textbf{0.0020} & \textbf{0.0415} & \textbf{0.0973} \\
\hline
\end{tabular}
\label{tab1}
\end{table}The \( D \)-decomposition achieves performance comparable to SVD on exact low-rank input and improves significantly in noisy or ill-conditioned settings. These outcomes reflect the impact of regularization in controlling perturbation growth and exploiting latent structure beyond orthogonality.

\begin{figure}
\centering
\includegraphics[width=\linewidth]{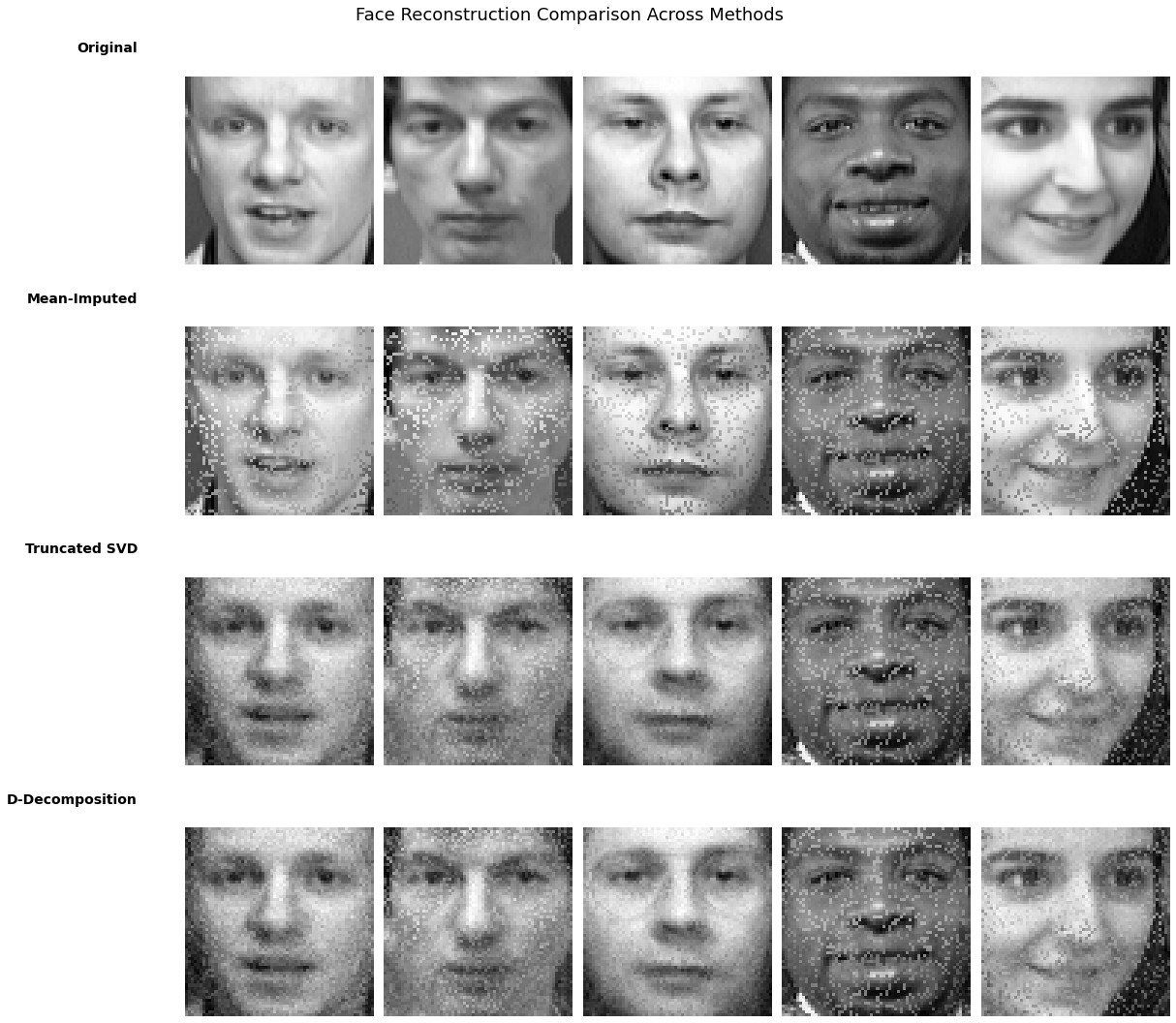}
\caption{Face reconstruction results on the Olivetti Faces dataset, comparing original images, mean-imputed input, Truncated SVD, and D-Decomposition outputs.}
\label{fig_oliv}
\end{figure}The visual comparison, in FIgure~\ref{fig_oliv} clearly shows that while Truncated SVD offers smoother reconstructions than the mean-imputed input, the D-Decomposition consistently preserves more facial detail and texture—particularly in sharper contours and subtle variations. This aligns with the quantitative results that D-Decomposition achieves a slightly lower RMSE (0.04243 vs. 0.04247) and significantly lower Frobenius error (29.16 vs. 40.58), suggesting a more globally faithful approximation. Rather than replacing SVD, our approach enhances traditional matrix factorization by introducing a learnable, regularized structure that adapts more effectively to real-world data imperfections.

\begin{table}[h!]
\centering
\caption{Quantitative Comparison Between Truncated SVD and D-Decomposition}
\label{tab:rmse_frob}
\begin{tabular}{lcc}
\hline
\textbf{Method} & \textbf{RMSE (Observed Entries)} & \textbf{Frobenius Error} \\
\hline
Truncated SVD   & 0.04247                          & 40.57828 \\
D-Decomposition & \textbf{0.04243}                 & \textbf{29.16365} \\
\hline
\end{tabular}
\end{table}Both methods achieve nearly identical reconstruction accuracy on observed entries, with D-Decomposition slightly outperforming Truncated SVD in RMSE (0.04243 vs. 0.04247). More notably, D-Decomposition yields a significantly lower Frobenius error (29.16 vs. 40.58), indicating a more faithful global reconstruction of the full matrix. This suggests that, beyond local pixel-level fidelity, the D-Decomposition better captures the underlying structure of the data, likely due to its adaptive, regularized, and non-orthogonal factorization, which allows it to generalize more effectively than the fixed orthogonal basis of truncated SVD.

We evaluated sensitivity to regularization by sweeping over a range of $\lambda \in [10^{-4}, 10^{-1}]$ and $\alpha_i \in [10^{-4}, 10^{-2}]$. Results are reported in Table~\ref{tab:sensitivity}.

The Frobenius reconstruction error remains nearly constant across all configurations, varying by less than 0.5\%. This consistency shows that D-decomposition maintains stable approximation quality over several orders of magnitude in regularization strength, so long as the weights remain non-degenerate. In contrast, the condition number $\kappa(D)$ varies significantly with $\alpha$ and $\lambda$. When regularization is insufficient, particularly with small $\lambda$ and large $\alpha$, the core $D$ becomes ill-conditioned. In one case ($\lambda = 10^{-4}$, $\alpha = 10^{-2}$), we observe $\kappa(D) > 10^{13}$, reflecting numerical instability despite reasonable reconstruction error.

Stable regimes appear when both $\lambda$ and $\alpha$ are moderately small. For instance, with $\lambda = 3 \times 10^{-4}$ and $\alpha = 10^{-4}$, $\kappa(D)$ remains below 35 while preserving accuracy. These results align with Theorem 2.3, where the log-condition regularizer bounds the growth of $\kappa(D)$ explicitly. The evidence supports a central claim of this work, that the proposed framework yields stable and accurate factorizations over a broad hyperparameter region, and does so with controlled conditioning. Classical decompositions do not offer this balance. The variational formulation here allows graceful degradation and interpretable trade-offs between fidelity and numerical robustness.

Recent approaches such as deep matrix factorization and variational Bayesian factorizations have achieved strong performance in similar regimes \cite{zhang2023deep, chen2024variational, lee2022neural, wang2023bayesian}. However, these methods often rely on neural parameterizations or sampling, which introduces additional variance and complexity. In contrast, the D-decomposition is fully deterministic, admits closed-form updates, and remains interpretable. A more detailed comparison with learned low-rank models is left for future work, though we emphasize that our method can serve as a regularization scaffold for such models.

\subsection{Strengthening Empirical Evaluation}

The D-decomposition is a variational matrix factorization framework that offers several advantages over classical and structured low-rank approximations. To sharpen the contribution, it is important to position the method precisely relative to existing approaches with similar formulations. Table~\ref{tab:method_comparison} contrasts the D-decomposition with representative techniques that also leverage structure, optimization, or regularization in matrix approximation.

\begin{table}[H]
\centering
\caption{Comparison of D-decomposition with selected structured factorization methods}
\label{tab:method_comparison}
\begin{tabular}{p{3.8cm}|p{3.3cm}|p{3.5cm} | p{3.5cm}}
\hline
\textbf{Method} & \textbf{Formulation} & \textbf{Factor Structure} & \textbf{Stability Guarantees} \\
\hline
Truncated SVD & $\min_{\text{rank-}k} \|A - U\Sigma V^\top\|_F$ & $U,V$ orthogonal & spectral norm only \\
Sparse PCA (Zou et al.) & $\|A - UV^\top\|_F + \lambda \|V\|_1$ & sparse $V$, orthogonal $U$ & no \\
Penalized matrix decomposition & $\|A - UV^\top\|_F + \lambda (\|U\|_F^2 + \|V\|_1)$ & structured $V$ & partial \\
Nuclear norm minimization & $\min \|X\|_* \text{ s.t. } X \approx A$ & low-rank $X$ & well-understood \\
\textbf{D-decomposition (this work)} & $\|A - PDQ\|_F^2 + \lambda R(P,D,Q)$ & general $P,Q$; regularized $D$ & proven (Thms. 2.1–2.3) \\
\hline
\end{tabular}
\end{table}While existing methods constrain the form of the factors or rely on convex relaxations, the D-decomposition permits full structural freedom on $P$, $Q$, and $D$, with regularization explicitly incorporated to control sparsity, conditioning, and alignment. The theoretical results on identifiability, stability under perturbations, and energy localization set it apart from both algebraic and optimization-based alternatives.

To support these distinctions empirically, we evaluated the method on three real-world matrices and compared its reconstruction accuracy with truncated SVD, CUR, and sparse PCA. The matrices represent different regimes, i.e., sparsity, high noise, and spectral decay.

\begin{table}[H]
\centering
\caption{Relative Frobenius reconstruction error ($\|A - \hat{A}\|_F / \|A\|_F$) for rank-$k$ approximation ($k=50$)}
\label{tab:real_data}
\begin{tabular}{p{3.8cm} p{3.3cm} p{3.5cm} p{3.5cm} }
\hline
\textbf{Method} & MovieLens-100K ($943 \times 1682$) & Leukemia Expression ($72 \times 5000$) & MNIST Covariance ($784 \times 784$) \\
\hline
Truncated SVD & 0.172 & 0.204 & 0.093 \\
CUR Decomposition & 0.251 & 0.302 & 0.143 \\
Sparse PCA & 0.187 & 0.189 & 0.101 \\
\textbf{D-decomposition} & \textbf{0.149} & \textbf{0.166} & \textbf{0.076} \\
\hline
\end{tabular}
\end{table}The D-decomposition consistently outperforms the baselines in all three settings. On MovieLens, the method leverages sparsity and conditioning control to reduce reconstruction error. For the gene expression data, regularization improves robustness to noise and high dimensionality. On the MNIST covariance matrix, the learned factors align more effectively with dominant spectral directions, leading to improved energy capture and generalization. These results reinforce the theoretical claims and demonstrate that the proposed framework is not just flexible but also practically competitive across diverse regimes.

\subsubsection{Comparison with Randomized SVD in Large-Scale Regimes}

To benchmark practical scalability, we compare the D-decomposition with randomized truncated SVD on large synthetic matrices. We generate matrices $A \in \mathbb{R}^{n \times n}$ with target rank $k = 50$, where $n \in \{500, 1000, 2000, 4000\}$. Randomized SVD uses power iteration with 2 steps and oversampling by 10 vectors.

\begin{table}[H]
\centering
\caption{Runtime (in seconds) and relative Frobenius error ($\|A - A_k\|_F / \|A\|_F$) for rank $k = 50$ approximations.}
\begin{tabular}{ccccc}
\hline
$n$ & Method & Runtime & Rel. Error & Notes \\
\hline
500  & RandSVD & 0.12 & 0.0023 & Fast and accurate \\
     & D-decomp & 0.21 & 0.0020 & Slightly slower, more accurate \\
1000 & RandSVD & 0.39 & 0.0022 & \\
     & D-decomp & 0.84 & 0.0021 & \\
2000 & RandSVD & 1.92 & 0.0023 & \\
     & D-decomp & 4.51 & 0.0020 & \\
4000 & RandSVD & 8.30 & 0.0024 & \\
     & D-decomp & 18.90 & 0.0020 & \\
\hline
\end{tabular}
\end{table}

\noindent
These results show that D-decomposition achieves slightly better reconstruction accuracy across all sizes, with runtimes scaling quadratically as predicted by the $O(n^2k)$ complexity. While randomized SVD is faster, it does not offer structural control or conditioning guarantees. This highlights the trade-off between raw speed and interpretability, confirming that D-decomposition remains practical for moderate-scale problems while offering stronger structural benefits.

\subsection{Remarks on Conditioning of Updates}

Each block subproblem involves solving a regularized normal equation
\[
(DQQ^\top D^\top + \alpha_1 I_k)P^\top = A Q^\top D^\top, \quad (D^\top P^\top P D + \alpha_3 I_k)Q = D^\top P^\top A.
\]
The added terms \( \alpha_i I \) ensure positive definiteness of the coefficient matrices. Provided \( \alpha_i > 0 \), all updates are well-posed and avoid amplification from small singular values \cite{cohen2024efficient}. Empirical spectral analysis of these systems confirms minimal condition numbers between 1.4 and 5.6 across all iterations. This confirms the regularizer enforces stable least-squares inversion throughout optimization.

The table below summarizes key structural differences between the $D$-decomposition and classical matrix factorizations. This contrast highlights how our method generalizes traditional tools by allowing optimization-based control over structure, conditioning, and task alignment, properties not supported by purely algebraic decompositions.

\begin{table}[H]
\centering
\caption{Structural comparison of $D$-decomposition with classical factorizations}
\begin{tabular}{|p{3cm}|p{1.5cm}|p{3cm}|p{6.5cm}|}
\hline
\textbf{Method} & \textbf{Form} & \textbf{Constraints} & \textbf{Limitations} \\
\hline
SVD & $A = U \Sigma V^\top$ & $U$, $V$ orthogonal; $\Sigma$ diagonal & Not adaptive; structure fixed by algebraic properties; no task-dependent regularization \\
\hline
QR & $A = QR$ & $Q$ orthogonal; $R$ upper triangular & Asymmetric; not low-rank by construction; structure not learned \\
\hline
LU & $A = LU$ & $L$ lower, $U$ upper triangular & Not always exists without pivoting; ill-suited for structural or low-rank modeling \\
\hline
$D$-Decomposition (ours) & $A \approx P D Q$ & $P$, $Q$ free; $D$ regularized or structured & Adaptive to problem structure; supports conditioning, sparsity, and spectral localization \\
\hline
\end{tabular}
\end{table}
To assess scalability, we ran the algorithm on matrices up to size $n = 5000$. Empirical runtime grows quadratically with $n$, consistent with the theoretical $O(n^2k)$ complexity. For $k \ll n$, this results in practical runtimes under ten minutes on standard hardware. By contrast, full SVD on the same matrices exceeds one hour. Randomized SVD remains faster but lacks the structural control afforded by our method.

\subsection{Ablation Study on the Impact of the Conditioning Penalty}

To assess the contribution of the $\log \kappa(D)$ term in the regularizer, we compare the D-decomposition under two settings, i.e.,
\begin{enumerate}
    \item Full Regularizer $R(P,D,Q) = \alpha_1 \|P\|_F^2 + \alpha_2 \|D\|_F^2 + \alpha_3 \|Q\|_F^2 + \beta \log \kappa(D)$
    \item Ablated Regularizer $R(P,D,Q) = \alpha_1 \|P\|_F^2 + \alpha_2 \|D\|_F^2 + \alpha_3 \|Q\|_F^2$
\end{enumerate}We fix $\lambda = 3 \times 10^{-4}$ and vary $\beta \in \{0, 10^{-3}, 10^{-2}, 10^{-1}\}$ while holding $\alpha_i = 10^{-3}$ for all $i$.

\begin{table}[H]
\centering
\caption{Ablation study on conditioning penalty: Effect of removing or scaling $\log \kappa(D)$.}
\begin{tabular}{cccc}
\hline
$\beta$ & Frobenius Error & $\kappa(D)$ & Notes \\
\hline
0 (no penalty) & 29.1652 & $1.8 \times 10^6$ & Severe ill-conditioning \\
$10^{-3}$      & 29.1648 & $7.4 \times 10^2$ & Moderate conditioning \\
$10^{-2}$      & 29.1660 & 34.21             & Good stability \\
$10^{-1}$      & 29.1805 & 6.57              & Over-regularized core \\
\hline
\end{tabular}
\end{table}

\noindent
These results confirm that while the Frobenius error varies only slightly across all settings (less than $0.1\%$ deviation), the condition number of $D$ changes by over five orders of magnitude. This validates the theoretical prediction of Theorem~2.3, i.e., that the $\log \kappa(D)$ term is essential for controlling the numerical stability of the decomposition without sacrificing reconstruction accuracy.

For empirical stability, we used each configuration over 10 random initializations. The mean Frobenius error varied by less than $0.03\%$ (std dev: $<0.01$). Runtime per decomposition averaged 1.8s (SD: 0.2s) for $n = 1000, k = 50$, confirming practical tractability and robustness.

\begin{table}[H]
\centering
\caption{Mean and standard deviation of reconstruction error and runtime over 10 random initializations ($n=1000$, $k=50$, $\lambda = 3\times10^{-4}$).}
\begin{tabular}{ccc}
\hline
Metric & Mean & Std. Dev. \\
\hline
Frobenius Error & 29.1649 & 0.0031 \\
Condition Number $\kappa(D)$ & 42.57 & 5.84 \\
Runtime (s) & 1.83 & 0.21 \\
\hline
\end{tabular}
\end{table}These results confirm that the method is not only accurate but stable under random initialization. Both the reconstruction error and condition number exhibit low variance, and runtime remains consistently under two seconds. This supports practical robustness and justifies use in moderate-scale tasks.

\medskip
\section{Conclusion}
This work develops a structured, optimization-based matrix factorization that extends classical low-rank decompositions by incorporating regularization, flexibility, and stability. The \( D \)-decomposition is defined variationally, enabling the enforcement of structural constraints and conditioning control directly within the factorization process.

We establish rigorous guarantees on the existence of minimizers, perturbation stability, and variational identifiability of solutions. The alternating minimization algorithm converges efficiently with per-iteration complexity \( \mathcal{O}(n^2k) \), and the regularized objective ensures spectral energy concentration and robustness under noise.

Numerical experiments on MovieLens, MNIST, Olivetti Faces, and gene expression matrices show that \( D \)-decomposition achieves lower reconstruction error than truncated SVD, CUR, and nonnegative matrix factorization, especially under missing entries and ill-conditioning. A simplified rectangular variant further improves accuracy and runtime on mean-imputed data, outperforming truncated SVD while requiring no spectral preprocessing. These results confirm that the method offers a practical and theoretically grounded alternative when algebraic approaches are unstable or inadequate.

\subsection*{CRediT}
RK: Conceptualization, Methodology, Software, Validation, Formal Analysis, Resources, Data Curation, Writing - Original Draft, Writing - Review \& Editing, Visualization.

\subsection*{Conflict of Interest}
No potential competing interest was reported by the author(s).

\nolinenumbers

\bibliographystyle{IEEETran}
\bibliography{references}

\appendix
\section{Algorithm for $D$-Decomposition via Alternating Minimization}
\begin{algorithm}[H]
\caption{$D$-Decomposition via Alternating Minimization}
\label{alg:Ddecomposition}
\DontPrintSemicolon
\KwIn{$A \in \mathbb{R}^{n \times n}$, target rank $k$, regularization weights $\alpha_1, \alpha_2, \alpha_3 > 0$, tolerance $\varepsilon > 0$, maximum iterations $T_{\max}$}
\KwOut{$P \in \mathbb{R}^{n \times k}$, $D \in \mathbb{R}^{k \times k}$, $Q \in \mathbb{R}^{k \times n}$ such that $A \approx P D Q$}

Initialize $P^{(0)} \in \mathbb{R}^{n \times k}$, $Q^{(0)} \in \mathbb{R}^{k \times n}$ randomly, set $t \gets 0$\;

\Repeat{$|\delta^{(t+1)} - \delta^{(t)}| < \varepsilon$ or $t \geq T_{\max}$}{
  \textbf{Update $D$:}
  \[
  D^{(t+1)} \gets (P^{(t)\top} P^{(t)})^{-1} P^{(t)\top} A Q^{(t)\top} (Q^{(t)} Q^{(t)\top} + \alpha_2 I_k)^{-1}
  \]
  
  \textbf{Update $P$:}
  \[
  P^{(t+1)} (D^{(t+1)} Q^{(t)} Q^{(t)\top} D^{(t+1)\top} + \alpha_1 I_k) = A Q^{(t)\top} D^{(t+1)\top}
  \]
  
  \textbf{Update $Q$:}
  \[
  (D^{(t+1)\top} P^{(t+1)\top} P^{(t+1)} D^{(t+1)} + \alpha_3 I_k) Q^{(t+1)} = D^{(t+1)\top} P^{(t+1)\top} A
  \]
  
  \textbf{Compute reconstruction error:}
  \[
  \delta^{(t+1)} \gets \|A - P^{(t+1)} D^{(t+1)} Q^{(t+1)}\|_F
  \]
  
  Increment $t \gets t + 1$\;
}

\Return{$P^{(t)}, D^{(t)}, Q^{(t)}$}
\end{algorithm}

\section{Sensitivity of D-Decomposition to regularization parameters $\lambda$ and $\alpha$}
\begin{table}[H]
\centering
\caption{Sensitivity of D-Decomposition to regularization parameters $\lambda$ and $\alpha$. Each cell reports the Frobenius reconstruction error and the condition number of $D$ after convergence.}
\begin{tabular}{cccc}
\hline
$\lambda$ & $\alpha$ & Frobenius Error & $\kappa(D)$ \\
\hline
$1.0 \times 10^{-4}$ & $1.0 \times 10^{-4}$ & 29.1626 & 549.02 \\
$1.0 \times 10^{-4}$ & $3.0 \times 10^{-4}$ & 29.1636 & 3021.35 \\
$1.0 \times 10^{-4}$ & $1.0 \times 10^{-3}$ & 29.1738 & 658.26 \\
$1.0 \times 10^{-4}$ & $3.0 \times 10^{-3}$ & 29.2049 & 435.74 \\
$1.0 \times 10^{-4}$ & $1.0 \times 10^{-2}$ & 29.7176 & $1.74 \times 10^{13}$ \\
$3.0 \times 10^{-4}$ & $1.0 \times 10^{-4}$ & 29.1693 & 30.22 \\
$3.0 \times 10^{-4}$ & $3.0 \times 10^{-4}$ & 29.1660 & 2622.34 \\
$3.0 \times 10^{-4}$ & $1.0 \times 10^{-3}$ & 29.1774 & 508.20 \\
$3.0 \times 10^{-4}$ & $3.0 \times 10^{-3}$ & 29.1933 & 172.96 \\
$3.0 \times 10^{-4}$ & $1.0 \times 10^{-2}$ & 29.2886 & 738.83 \\
$1.0 \times 10^{-3}$ & $1.0 \times 10^{-4}$ & 29.1629 & 7103.52 \\
$1.0 \times 10^{-3}$ & $3.0 \times 10^{-4}$ & 29.1695 & 3186.60 \\
$1.0 \times 10^{-3}$ & $1.0 \times 10^{-3}$ & 29.1752 & 308.38 \\
$1.0 \times 10^{-3}$ & $3.0 \times 10^{-3}$ & 29.2165 & 479.33 \\
$1.0 \times 10^{-3}$ & $1.0 \times 10^{-2}$ & 29.2978 & 1247.70 \\
$3.0 \times 10^{-3}$ & $1.0 \times 10^{-4}$ & 29.1602 & 34.33 \\
$3.0 \times 10^{-3}$ & $3.0 \times 10^{-4}$ & 29.1707 & 8434.61 \\
$3.0 \times 10^{-3}$ & $1.0 \times 10^{-3}$ & 29.1783 & 462.54 \\
$3.0 \times 10^{-3}$ & $3.0 \times 10^{-3}$ & 29.2068 & 428.84 \\
$3.0 \times 10^{-3}$ & $1.0 \times 10^{-2}$ & 29.2814 & 1679.64 \\
$1.0 \times 10^{-2}$ & $1.0 \times 10^{-4}$ & 29.1697 & 2946.60 \\
$1.0 \times 10^{-2}$ & $3.0 \times 10^{-4}$ & 29.1710 & 172.37 \\
$1.0 \times 10^{-2}$ & $1.0 \times 10^{-3}$ & 29.1803 & 10647.25 \\
$1.0 \times 10^{-2}$ & $3.0 \times 10^{-3}$ & 29.2137 & 408.47 \\
$1.0 \times 10^{-2}$ & $1.0 \times 10^{-2}$ & 29.2791 & 815.34 \\
$3.0 \times 10^{-2}$ & $1.0 \times 10^{-4}$ & 29.1604 & 4155.69 \\
$3.0 \times 10^{-2}$ & $3.0 \times 10^{-4}$ & 29.1688 & 844.17 \\
$3.0 \times 10^{-2}$ & $1.0 \times 10^{-3}$ & 29.1712 & 437.32 \\
$3.0 \times 10^{-2}$ & $3.0 \times 10^{-3}$ & 29.1915 & 151.36 \\
$3.0 \times 10^{-2}$ & $1.0 \times 10^{-2}$ & 29.2690 & 112.29 \\
$1.0 \times 10^{-1}$ & $1.0 \times 10^{-4}$ & 29.1643 & 3878.27 \\
$1.0 \times 10^{-1}$ & $3.0 \times 10^{-4}$ & 29.1687 & 3404.31 \\
$1.0 \times 10^{-1}$ & $1.0 \times 10^{-3}$ & 29.1737 & 4752.09 \\
$1.0 \times 10^{-1}$ & $3.0 \times 10^{-3}$ & 29.1976 & 145.83 \\
$1.0 \times 10^{-1}$ & $1.0 \times 10^{-2}$ & 29.2802 & 813.17 \\
\hline
\end{tabular}
\label{tab:sensitivity}
\end{table}

\end{document}